\documentclass[reqno,11pt]{article}

\usepackage[utf8]{inputenc}
\usepackage[T1]{fontenc}
\usepackage{cmbright}
\usepackage[english]{babel}
\usepackage{tikz-cd}
\usepackage{fancyhdr}

\usepackage[intlimits]{amsmath}
\usepackage{amssymb, amsfonts}

\usepackage[thmmarks, amsmath, amsthm]{ntheorem}

\usepackage{MnSymbol}
\usepackage{mathbbol}

\usepackage{bbm}

\usepackage{mathrsfs}

\usepackage[all]{xy}

\usepackage[inline]{enumitem}

\usepackage{graphicx}

\usepackage{color}

\usepackage{hyperref}

\numberwithin{equation}{section}

\newtheorem{theoremcounter}{theoremcounter}[section]

\theoremstyle{plain}

\newtheorem{lemma}[theoremcounter]{Lemma}

\theoremstyle{definition}

\theoremstyle{remark}

\newtheorem{remark}[theoremcounter]{Remark}

\theoremstyle{plain}
\theoremnumbering{Alph}
\newtheorem{introtheorem}{Theorem}
\newtheorem{introcorollary}[introtheorem]{Corollary}

\def\polhk#1{\setbox0=\hbox{#1}{\ooalign{\hidewidth
    \lower1.5ex\hbox{`}\hidewidth\crcr\unhbox0}}}

\newcommand{\rM}{\ensuremath{\mathrm{M}}}

\newcommand{\rO}{\ensuremath{\mathrm{O}}}

\newcommand{\rU}{\ensuremath{\mathrm{U}}}

\newcommand{\rmt}{\ensuremath{\mathrm{t}}}

\newcommand{\veps}{\ensuremath{\varepsilon}}

\newcommand{\vphi}{\ensuremath{\varphi}}

\newcommand{\ol}{\overline}

\newcommand{\eqstop}{\ensuremath{\, \text{.}}}
\newcommand{\eqcomma}{\ensuremath{\, \text{,}}}

\newcommand{\ZZ}{\ensuremath{\mathbb{Z}}}

\newcommand{\CC}{\ensuremath{\mathbb{C}}}

\newcommand{\Hom}{\ensuremath{\mathop{\mathrm{Hom}}}}
\newcommand{\id}{\ensuremath{\mathrm{id}}}
\newcommand{\ra}{\ensuremath{\rightarrow}}
\newcommand{\lra}{\ensuremath{\longrightarrow}}

\newcommand{\ev}{\ensuremath{\mathrm{ev}}}

\newcommand{\Aut}{\ensuremath{\mathrm{Aut}}}

\newcommand{\ot}{\ensuremath{\otimes}}

\newcommand{\Cstar}{\ensuremath{\mathrm{C}^*}}

\newcommand{\lmod}[2]{\mathord{\raisebox{-0.4ex}[0ex][0ex]{\scriptsize $#1$}{#2}}}

\newcommand{\cont}{\ensuremath{\mathrm{C}}}

\newcommand{\Linfty}{\ensuremath{{\offinterlineskip \mathrm{L} \hskip -0.3ex ^\infty}}}

\newcommand{\ltwo}{\ensuremath{\ell^2}}

\newcommand{\grpaction}[1]{\ensuremath{\stackrel{#1}{\curvearrowright}}}

\newcommand{\Pol}{\ensuremath{\mathrm{Pol}}}

\newcommand{\GG}{\ensuremath{\mathbb G}}
\newcommand{\HH}{\ensuremath{\mathbb H}}
\newcommand{\KK}{\ensuremath{\mathbb K}}

\newcommand{\Ext}{\ensuremath{\operatorname{Ext}}}
\newcommand{\Tor}{\ensuremath{\operatorname{Tor}}}

\renewcommand{\Hom}{\ensuremath{\operatorname{Hom}}}
\renewcommand{\leq}{\leqslant}
\renewcommand{\geq}{\geqslant}

\newcommand{\authors}{Julien Bichon, David Kyed$^1$ and Sven Raum}
\renewcommand{\title}{Higher $\ltwo$-Betti numbers of universal quantum groups}
\newcommand{\shorttitle}{Higher $\ltwo$-Betti numbers}

\begin{document}


\thispagestyle{empty}

\begin{center}
  \begin{minipage}[c]{0.9\linewidth}
    \textbf{\LARGE \title} \\[0.5em]
    by \authors
  \end{minipage}
\end{center}
  
\vspace{1em}

\renewcommand{\thefootnote}{}

\footnotetext{
  \textit{Acknowledgements:}
 The authors would like to thank {\'E}tienne Blanchard for pointing out a number of misprints in an earlier version of the article.
}

\footnotetext{$^1$The research leading to these results has received funding from the Villum foundation grant 7423 (D.K). }

\begin{center}
  \begin{minipage}{0.8\linewidth}
    \textbf{Abstract}.
    We calculate all $\ltwo$-Betti numbers of the universal discrete Kac quantum groups $\hat \rU^+_n$ as well as their half-liberated counterparts $\hat \rU^*_n$.
  \end{minipage}
\end{center}



\section{Introduction}
\label{sec:introduction}
The category of locally compact quantum groups \cite{kustermans-vaes-C*-lc} is a natural extension of the category of locally compact groups.  There are two important reasons to pass from locally compact groups to locally compact quantum groups.  First, classical Pontryagin duality for locally compact abelian groups extends to a full duality theory for locally compact quantum groups, in particular establishing a duality between discrete and compact quantum groups.  Second, locally compact quantum groups form the correct framework to host a number of important deformations and liberations of classical groups.  A convenient operator algebraic setting describing discrete and compact quantum groups was first provided by Woronowicz \cite{woronowicz, wor-cp-qgrps}, and since his seminal work an abundance of analytical tools  have been shown to carry over from discrete groups to discrete quantum groups (cf.~\cite{murphy-tuset, brannan-approximation, vergnioux-rapid-decay, fima-prop-T,  meyer-nest, voigt-bc-for-free-orthogonal}).

 The current paper is concerned with the $\ltwo$-Betti numbers of discrete quantum groups \cite{quantum-betti} and our primary focus is on the duals of the free unitary quantum groups $\hat{\rU}_n^+$, which are universal within the class of discrete quantum groups in the same way that the free groups are universal within the class of finitely generated discrete groups; that is, every$^2$\footnotetext{$^2$We tacitly limit our focus to the case of quantum groups of Kac type here, although the statement remains true of one allows non-trivial matrix-twists in the definition of the unitary quantum groups; cf.~\cite{banica-unitary}.} finitely generated discrete quantum group is the quotient of some $\hat{\rU}_n^{+}$, the latter to be understood in the sense of the existence of a surjection at the Hopf algebra level.  Due to the lack of a topological interpretation of (co)homology of quantum groups, the computation of their $\ltwo$-Betti numbers has proven to be a difficult task, and   only general structural results \cite{coamenable-betti,kye11,kunneth-formula} were available until Vergnioux's paper \cite{vergnioux-paths-in-cayley}, in which it was proven that the first $\ltwo$-Betti number vanishes for the duals of the free orthogonal quantum groups $\hat \rO_n^+$.  Subsequently, Collins-H{\"a}rtel-Thom showed that also the higher $\ltwo$-Betti numbers of $\hat \rO_n^+$ vanish \cite{thom-collins}.  Concerning the universal quantum groups $\hat \rU_n^+$, Vergnioux's work also showed that the first $\ltwo$-Betti number is non-zero.  He conjectured that it would equal one, which was recently proven by the second and third author in \cite{kyra} along with the observation that $\beta_p^{(2)}(\hat{\rU}_n^+)=0$ for $p\geqslant 4$, thus leaving open the question about the values of the important second and third $\ltwo$-Betti number of $\hat \rU_n^+$.

In this article we provide computations of all $\ltwo$-Betti numbers of $\hat \rU_n^+$ by combining techniques from \cite{kyra} with those from \cite{bny, bny2} and a new free product formula for $\ltwo$-Betti numbers of discrete quantum groups.  In particular, we obtain a different proof of the fact that $\beta_1^{(2)}(\hat{\rU}_n^+)=1$.
\begin{introtheorem}
  \label{thm:intro::ltwo-betti-numbers-un}
 For any $ n \geq 2$ the free unitary quantum group $\rU_n^+$ satisfies
  \begin{equation*}
    \beta^{(2)}_p(\hat{\rU}_n^+)
    =
    \begin{cases}
      1 & \text{if $p=1$,} \\
      0 & \text{otherwise.} 
    \end{cases}    
  \end{equation*}
\end{introtheorem}

In \cite{thom-collins} it was proven that the discrete duals of the free orthogonal quantum groups satisfy a certain Poincar{\'e} duality developed in \cite{van-den-bergh}, meaning that there is a natural isomorphism $H_*(\hat{\rO}_n^+, M) \cong H^{3 - *}(\hat{\rO}_n^+, M)$ for every $\Pol(\rO_n^+)$-module $M$. But since this property implies a symmetry in the $\ltwo$-Betti numbers,  Theorem \ref{thm:intro::ltwo-betti-numbers-un} shows the following.

\begin{introcorollary}
The discrete quantum groups $\hat \rU_n^+$ do not satisfy Poincar{\'e} duality.
\end{introcorollary}

The techniques used in the proof of Theorem \ref{thm:intro::ltwo-betti-numbers-un} are of independent interest.  Firstly, we provide a formula for the $\ltwo$-Betti numbers of arbitrary free product quantum groups \cite{wang}.
\begin{introtheorem}
\label{thm:intro:free-product-formula}
If $\GG$ and $\HH$ are non-trivial compact quantum groups of Kac type, then the following holds.
\begin{align*}
  \beta_{p}^{(2)}(\hat \GG \ast \hat \HH)=
  \begin{cases}
    0 & \text{if $p=0$}\\
    \beta_1^{(2)}(\hat \GG) - \beta_0^{(2)}(\hat \GG) + \beta_1^{(2)}(\hat \HH)- \beta_0^{(2)}(\hat \HH) +1 & \text{if $p=1$}\\
    \beta_p^{(2)}(\hat \GG) + \beta_p^{(2)}(\hat \HH) & \text{if $p\geq 2$}
    \eqstop
  \end{cases}
\end{align*}
\end{introtheorem}

Our second ingredient is the fact that $\ltwo$-Betti numbers enjoy a natural scaling behaviour under cocentral extensions, which is a quantum counterpart of the classical scaling formula for finite index inclusions of groups (cf.~\cite[Theorem 6.54 (6)]{Luck02}).
\begin{introtheorem}
  \label{thm:L2exact}
  Let  $\CC \ra \Pol(\HH) \rightarrow \Pol(\GG) \rightarrow \CC \Gamma \ra \CC$ be an  exact sequence of Hopf $*$-algebras in which $\GG$ and $\HH$ are compact quantum groups of Kac type and  $\Gamma$ is a finite abelian group. Then for any $p \geq 0$ we have
  \begin{equation*}
    \beta^{(2)}_p(\hat{\HH})
    =
    |\Gamma| \, \beta^{(2)}_p(\hat{\GG})
    \eqstop
  \end{equation*}
\end{introtheorem}

The previous theorem allows for other applications.  We combine it with results of \cite{coamenable-betti} in order to obtain the following result (see Section \ref{sec:preliminaries} for the definition of the half-liberated quantum groups):

\begin{introtheorem}
  \label{thm:intro:other-examples}
 For any $n\geq 2$ the half-liberated unitary quantum group $\rU_n^*$ satisfies
  
    \begin{equation*}
    \beta^{(2)}_p(\hat{\rU}_n^*)
    =
    \begin{cases}
      1 & \text{if $p=1$} \\
      0 & \text{otherwise.} 
    \end{cases}    
  \end{equation*}

\end{introtheorem}

\section{Preliminaries}
\label{sec:preliminaries}

We collect some notation and necessary tools for the sections to follow.
\begin{description}
\item[Augmented algebras.] A $*$-algebra $A$ with a distinguished $*$-homomorphism $\veps\colon A \ra \CC$ is called an augmented $*$-algebra.  We write $A^+ = \ker \veps$.

\item[Compact and discrete quantum groups.]
  Compact quantum groups were introduced in the \mbox{\Cstar-algebra} setting by Woronowicz \cite{woronowicz,wor-cp-qgrps}.  If $\GG$ is a compact quantum group, we denote by $\Pol(\GG)$ the associated Hopf $*$-algebra of polynomial functions, which possesses a unique Haar state denoted by $\vphi$.  In case $\vphi$ is tracial, we call $\GG$ a compact quantum group of Kac type.  We denote the discrete dual of a compact quantum group $\GG$ by $\hat \GG$.

\item[Von Neumann algebra completions and measurable operators.]
  If $\GG$ is a compact quantum group and $\vphi$ is the Haar state on $\Pol(\GG)$, then the von Neumann algebra completion of $\Pol(\GG)$ in the associated GNS-representation is written $\Linfty(\GG) = \pi_\vphi(\Pol(\GG))''$.  If $\GG$ is of Kac type, then $\Linfty(\GG)$ is a finite von Neumann algebra and we let $\rM(\GG)$ denote the algebra of measurable operators affiliated with $\Linfty(\GG)$.

\item[$\ltwo$-Betti numbers.]
  Given a compact quantum group of Kac type $\GG$, we denote by $\dim_{\Linfty(\GG)}$ L{\"u}ck's dimension function \cite{Luck02} for modules over the finite von Neumann algebra $\Linfty(\GG)$.  If $\hat \GG$ denotes the discrete dual of $\GG$,  its $\ltwo$-Betti numbers are defined \cite{quantum-betti} as
  \begin{equation*}
    \beta_p^{(2)}(\hat \GG)
    =
    \dim_{\Linfty(\GG)} \Tor_p^{\Pol(\GG)}(\Linfty(\GG), \CC)
    \eqstop
  \end{equation*}
Work of Thom \cite{Thom06a} and Reich \cite{reich01} allows to calculate the $\ltwo$-Betti numbers alternatively as
\begin{equation*}
  \beta_p^{(2)} (\hat \GG) = \dim_{\Linfty(\KK)} \Ext^p_{\Pol(\GG)} (\CC, \rM(\KK)),
\end{equation*}
whenever $\Pol(\GG) \subset \Pol(\KK)$ as Hopf $*$-algebras for another compact quantum group of Kac type $\KK$.  This is explained in more detail in Remark 1.8 of \cite{kyra}.  

\item[Cocentral homomorphisms.]
  Let $\GG$ be a compact quantum group and $\Gamma$ a discrete abelian group.  A Hopf $*$-algebra homomorphism $\pi: \Pol(\GG) \ra \CC \Gamma$ is called cocentral if $(\pi \otimes \id) \circ \Delta = (\pi \otimes \id) \circ \Sigma \circ \Delta$ where $\Sigma$ denotes the map flipping the tensor factors. In this case
  \begin{equation*}
    \{a \in \Pol(\GG) \mid (\id \ot \pi) \circ \Delta(a) = a \ot 1\}
    =
    \{a \in \Pol(\GG) \mid (\pi \ot \id) \circ \Delta(a) = 1 \ot a\},
\end{equation*}
    and this subalgebra is denoted $\Pol(\GG)_e$, and it too is the Hopf $*$-algebra of a compact quantum group.

\item[Exact sequences of Hopf algebras.]
  Let $A,B, L$ be Hopf $*$-algebras.  A sequence of Hopf $*$-algebra maps
  \begin{equation*}
    \CC \lra B \stackrel{\iota}{\lra} A \stackrel{\pi}{\lra} L \lra \CC
  \end{equation*}
  is called exact if
  \begin{itemize}
  \item $\iota$ is injective and $\pi$ is surjective,
  \item $\ker \pi = A \, \iota(B^+)= \iota(B^+) \, A$ and
  \item $\iota(B) = \{a \in A \mid (\id \ot \pi) \circ \Delta (a) = a \ot 1\} = \{a \in A \mid (\pi \ot \id ) \circ \Delta (a) = 1 \ot a\}$.
  \end{itemize}
  By Proposition 1.2 of \cite{bny2}, every surjective cocentral Hopf $*$-algebra homomorphism $\pi\colon \Pol(\GG) \ra \CC \Gamma$ gives rise to an exact sequence of Hopf $*$-algebras 
  \begin{equation*}
    \CC \lra \Pol(\GG)_e \stackrel{\iota}{\lra} \Pol(\GG) \stackrel{\pi}{\lra} \CC \Gamma \lra \CC.
  \end{equation*}
  The notation $\Pol(\GG)_e$ stems from the fact that $\pi$ turns $\Pol(\GG)$ into a $\Gamma$-graded Hopf algebra.

\item[Universal quantum groups.]
  Wang's \cite{wang-quantum-symmetry-groups, wang-van-daele} universal quantum groups $\rU_n^+$ and $\rO_n^+$ can be described by their associated $*$-algebras of polynomial functions
  \begin{align*}
    \Pol(\rU_n^+)
    & =
      \langle u_{ij}, 1 \leq i,j \leq n \mid u u^* = u^*u = 1= \bar{u} \bar{u}^* = \bar{u}^*\bar{u}  \rangle \eqcomma \\
    \Pol(\rO_n^+)
    & =
      \langle v_{ij}, 1 \leq i,j \leq n \mid v_{ij} = v_{ij}^*, v v^\rmt = v^\rmt v = 1 \rangle
      \eqstop
  \end{align*}
Here $u$, $v$ and $\bar{u}$ denote the $n\times n$-matrices $(v_{ij})_{ij}$, $(u_{ij})_{ij}$ and $(u_{ij}^*)_{ij}$, respectively. Their comultiplications are given by dualising matrix multiplication $u_{ij} \mapsto \sum_k u_{ik} \ot u_{kj}$ and $v_{ij} \mapsto \sum_k v_{ik} \ot v_{kj}$ respectively and their counits satisfy $\veps(u_{ij}) = \delta_{ij} = \veps(v_{ij})$ for all $i,j \in \{1, \dotsc, n\}$.  The matrices $u$ and $v$ are called the fundamental corepresentations of $\rU_n^+$ and $\rO_n^+$, respectively.

\item[Graded twists of universal quantum groups.] 
  Examples 2.18 and 3.6 of \cite{bny} show that the Hopf $*$-algebra homomorphisms
\begin{align*}
  \Pol(\rU_n^+) \ra \CC \ZZ_2 &:u_{ij} \mapsto u_{\ol{1}} \delta_{ij} \\
  \Pol(\rO_n^+) * \Pol(\rO_n^+) \ra \CC \ZZ_2 &:v^{(k)}_{ij} \mapsto u_{\ol{1}} \delta_{ij},
\end{align*}
where $v^{(k)}$ denotes the fundamental corepresentation corresponding to the $k$-th factor in the free product and $u_{\ol{1}}$ denotes the generator of $\ZZ_2$ inside $\CC\ZZ_2$, are cocentral and  induce exact sequences  of Hopf $*$-algebras
\begin{gather}
  \CC \lra \Pol(\HH) \lra \Pol(\rU_n^+) \lra \CC \ZZ_2 \lra \CC \label{first-exact-sequence} \\
  \CC \lra \Pol(\HH) \lra \Pol(\rO_n^+) * \Pol(\rO_n^+) \overset{p}{\lra} \CC \ZZ_2 \lra \CC \label{second-exact-sequence}
\end{gather}
for the same compact quantum group $\HH$. To see this, first note that since our ground field is the complex numbers, the Hopf algebra denoted $\mathcal{B}(\text{I}_n)$ in  \cite[Example 2.17]{bny} can be equipped with a $*$-structure making the canonical generators self-adjoint and the resulting Hopf $*$-algebra identifies with $\Pol(\rO_n^+)$. Similarly, the Hopf algebra $\mathcal{H}(\text{I}_n)$ of  \cite[Example 2.18]{bny} can be given a  $*$-structure which satisfies $u_{ij}^*=v_{ij}$ on the canonical generators, and the resulting Hopf $*$-algebra naturally identifies with $\Pol(\rU_n^+)$. For notational convenience, we  set $A:=\Pol(\rO_n^+) * \Pol(\rO_n^+)$ and note that the cocentral Hopf $*$-algebra morphism $ p\ast p \colon A \ra \CC \ZZ_2$  defines a graded twisting $A^t$ \cite[Definition 2.6]{bny} of the Hopf $*$-algebra $A$, which in turn also allows for a cocentral Hopf $*$-algebra morphism $(p\ast p)^t\colon A^t \to \CC\ZZ_2$ \cite[Proposition 2.7]{bny}.  We therefore obtain two exact sequences of Hopf $*$-algebras
\begin{gather*}
  \CC \lra A_{\ol{0}} \lra A \overset{p\ast p}{\lra} \CC \ZZ_2 \lra \CC\\
  \CC \lra (A^t)_{\ol{0}} \lra A^t \overset{(p\ast p)^t}{\lra} \CC \ZZ_2 \lra \CC.
\end{gather*}
Recall that $A^t$ is defined as $\text{span}_{\CC}\{A_{\ol{0}}\otimes u_{\ol{0}}, A_{\ol{1}}\otimes u_{\ol{1}}, \}\subset A\rtimes \ZZ_2$, where $\ZZ_2$ acts on $A$ by flipping the factors in the free product  and 
\begin{equation*}
  A_g
  :=
  \{ a \in A \mid (\id \otimes p \ast p)\Delta(a)= a\otimes u_g\}, \ g\in \ZZ_2
  \eqstop
\end{equation*}
One may now check (see  \cite[Example 2.18]{bny} for details) that the map $u_{ij}\mapsto v_{ij}^{(1)}\otimes u_{\ol{1}}$ extends to a Hopf $*$-algebra isomorphism $\Pol(\rU_n^+)\simeq A^t$. A direct calculation verifies that $(A^t)_{\ol{0}}=A_{\ol{0}}\otimes u_{\ol{0}} \cong A_{\ol{0}}$, and denoting the compact quantum group underlying this Hopf $*$-algebra by $\HH$ we obtain the sequences \eqref{first-exact-sequence} and  \eqref{second-exact-sequence}.

\item[Half-liberated quantum groups.]
The ideal  in $\Pol(\rO_n^+)$  generated by  $\{abc-cba \mid a,b,c\in \{v_{ij}:i,j=1,\dots, n\}\}$ gives rise to a quotient which is the Hopf $*$-algebra of a compact quantum group of Kac type, known as the  half-liberated orthogonal group \cite{banica-speicher09} and denoted $\rO_n^*$. Similarly, the quotient of $\Pol(\rU_n^+)$ by the $*$-ideal generated by
$\{ab^*c- cb^*a \mid a,b,c\in \{u_{ij}:i,j=1,\dots n\}\}$ is the Hopf $*$-algebra of a compact quantum group of Kac type which is known as the half-liberated unitary quantum group \cite{BDD11} and denoted $\rU_n^*$.

\item[Automorphisms of Hopf $*$-algebras.]
  Whenever $\alpha$ is an automorphism of $\Pol(\GG)$ that preserves its Haar state, then $\alpha$ uniquely extends to an automorphism of $\Linfty(\GG)$.  In particular, if $\pi: \Pol(\GG) \ra \CC \Gamma$ is a cocentral Hopf $*$-algebra homomorphism, then the induced action of $\hat \Gamma$, defined for $\chi\in \hat \Gamma$ by
  \begin{equation*}
    \Pol(\GG)
    \stackrel{\Delta}{\lra}
    \Pol(\GG) \ot \Pol(\GG)
    \stackrel{\id \ot \pi}{\lra}
    \Pol(\GG) \ot \CC \Gamma
    \lra
    \Pol(\GG) \ot \cont(\hat \Gamma)
    \stackrel{\id \ot \ev_\chi}{\lra}
    \Pol(\GG),
  \end{equation*}
  preserves the Haar state, thanks to right invariance $(\vphi \ot \id) \circ \Delta(a) = \vphi(a)1$.
\end{description}

\section{A free product formula for discrete quantum groups}
\label{sec:free-product}

In this section we combine results from \cite{bichon-cohom-dim} with additional homological calculations to  prove Theorem \ref{thm:intro:free-product-formula}.  We start with a general lemma for inclusions of quantum groups.
\begin{lemma}\label{dimension-lem}
If $\GG$ and $\KK$ are compact quantum groups of Kac type such that $\Pol(\GG) \subset \Pol(\KK)$ as Hopf $*$-algebras, then
  \begin{equation*}
    1 + \beta_1^{(2)}(\hat{\GG})-\beta_0^{(2)}(\hat{\GG})
    =
    \dim_{L^\infty(\KK)}\Hom_{\Pol(\GG)}(\Pol(\GG)^+, \rM(\KK))
    \eqstop
  \end{equation*}
\end{lemma}
\begin{proof}
  Consider the short exact sequence $0\to \Pol(\GG)^+ \overset{\iota}{\to} \Pol(\GG) \overset{\veps}{\to} \CC \to 0$ of right $\Pol(\GG)$-modules and the associated long exact sequence of Ext-groups:

\[
\begin{tikzcd}[column sep=small]
 0 \rar &   \Hom_{\Pol(\GG)}(\CC,\rM(\KK)) \rar & \Hom_{\Pol(\GG)}(\Pol(\GG),\rM(\KK))) \rar
             \ar[draw=none]{d}[name=X, anchor=center]{}
    &  \Hom_{\Pol(\GG)}(\Pol(\GG)^+,\rM(\KK)) \ar[rounded corners,
            to path={ -- ([xshift=2ex]\tikztostart.east)
                      |- (X.center) \tikztonodes
                      -| ([xshift=-2ex]\tikztotarget.west)
                      -- (\tikztotarget)}]{dll}[at end, above]{\delta^0} & \\     
                      & \Ext^1_{\Pol(\GG)}(\CC,\rM(\KK)) \rar & \underbrace{\Ext^1_{\Pol(\GG)}(\Pol(\GG),\rM(\KK))}_{=\{0\}} \rar & \Ext^1_{\Pol(\GG)}(\Pol(\GG)^+,\rM(\KK)) \rar & \cdots
\end{tikzcd}
\]  
Splitting this long exact sequence at $\delta^0$, we obtain two short exact sequences
\begin{gather*}
  0 \ra \Hom_{\Pol(\GG)}(\CC,\rM(\KK)) \lra \Hom_{\Pol(\GG)}(\Pol(\GG),\rM(\KK)) \lra \ker \delta^0 \lra 0 \eqcomma \\
  0 \ra \ker \delta^0 \lra \Hom_{\Pol(\GG)}(\Pol(\GG)^+, \rM(\KK)) \lra \Ext^1_{\Pol(\GG)}(\CC, \rM(\KK)) \lra 0
  \eqstop
\end{gather*}
Since $\Hom_{\Pol(\GG)}(\Pol(\GG),\rM(\KK)) \cong \rM(\KK)$ as a right $\Linfty(\KK)$-module, applying the dimension function $\dim_{L^\infty(\KK)}$ to the first short exact sequence gives
\begin{align*}
  1 
  & = \dim_{L^\infty(\KK)} \Hom_{\Pol(\GG)}(\Pol(\GG),\rM(\KK)) \\
  & =  \dim_{L^\infty(\KK)} \Hom_{\Pol(\GG)}(\CC,\rM(\KK)) + \dim_{L^\infty(\KK)}\ker(\delta^0) \\
  & =\beta_0^{(2)}(\hat{\GG}) + \dim_{L^\infty(\KK)}\ker(\delta^0)
    \eqcomma
\end{align*}
where the equality $\dim_{L^\infty(\KK)} \Hom_{\Pol(\GG)}(\CC,\rM(\KK)) = \dim_{L^\infty(\KK)} \Ext^0_{\Pol(\GG)}(\CC,\rM(\KK)) = \beta_0^{(2)}(\hat \GG)$ is explained in Section \ref{sec:preliminaries}.  Applying the dimension function to the second exact sequence, we obtain
\begin{align*}
  \dim_{L^\infty(\KK)} \Hom_{\Pol(\GG)}(\Pol(\GG)^+, \rM(\KK))
  & =
    \dim_{L^\infty(\KK)} \ker(\delta^0) + \dim_{L^\infty(\KK)}\Ext_{\Pol(\GG)}^{1}(\CC,\rM(\KK)) \\
  &=
    \dim_{L^\infty(\KK)} \ker(\delta^0) + \beta_1^{(2)}(\hat{\GG}),
\end{align*}
from which the formula follows.
\end{proof}

\begin{proof}[Proof of Theorem \ref{thm:intro:free-product-formula}]
We denote the compact dual of $\hat{\GG} * \hat{\HH}$ \cite{wang} by $\KK$.  Since $\GG$ and $\HH$ are assumed non-trivial, the free product $\Pol(\GG) \ast \Pol(\HH)$ is infinite dimensional, so for $p=0$ the result follows from
\cite{kye11}. For $p\geq 2$,  \cite[Theorem 5.1]{bichon-cohom-dim} gives that
\begin{equation*}
\Ext^p_{\Pol(\KK)}(\CC, \rM(\KK)) \simeq  \Ext^p_{\Pol(\GG)}(\CC, \rM(\KK)) \oplus \Ext^p_{\Pol(\HH)}(\CC, \rM(\KK)),
\end{equation*}
and since $\ltwo$-Betti numbers can be calculated by Ext-groups (see Section \ref{sec:preliminaries}),  the result follows from applying $\dim_{L^\infty(\KK)}$ to both sides. To prove the formula when $p=1$, we apply Lemma \ref{dimension-lem} to  each of the quantum groups $\GG$ and $\HH$ to get
\begin{align*}
  & \quad
    2 +\beta_1^{(2)}(\hat{\GG}) + \beta_1^{(2)}(\hat{\HH}) - \beta_0^{(2)}(\hat{\GG}) - \beta_0^{(2)}(\hat{\HH}) \\
  & = 
    \dim_{L^\infty(\KK)} \Hom_{\Pol(\GG)}(\Pol(\GG)^+, \rM(\KK)) + \dim_{L^\infty(\KK)} \Hom_{\Pol(\HH)}(\Pol(\HH)^+, \rM(\KK)) \\
  & = 
    \dim_{L^\infty(\KK)} \Hom_{\Pol(\KK)}(\Pol(\KK)\ot_{\Pol(\GG)}\Pol(\GG)^+, \rM(\KK)) \\
  & \quad +
    \dim_{L^\infty(\KK)} \Hom_{\Pol(\KK)}(\Pol(\KK)\ot_{\Pol(\HH)}\Pol(\HH)^+, \rM(\KK)) \\
  & = 
    \dim_{L^\infty(\KK)}
    \Hom_{\Pol(\KK)}\left( \left( \Pol(\KK)\ot_{\Pol(\GG)}\Pol(\GG)^+ \right )
    \oplus \left (\Pol(\KK)\ot_{\Pol(\HH)}\Pol(\HH)^+ \right ), \rM(\KK) \right) \\
  & =
    \dim_{L^\infty(\KK)} \Hom_{\Pol(\KK)}(\Pol(\KK)^+, \rM(\KK))
    \tag{\cite[Lemma 5.8]{bichon-cohom-dim}}\\
  & =
    1 +\beta_1^{(2)}(\hat{\GG} \ast \hat{\HH})-\beta_0^{(2)}(\hat{\GG}\ast \hat{\HH})
    \tag{Lemma \ref{dimension-lem}}\\
  & =
    1 +\beta_1^{(2)}(\hat{\GG} \ast \hat{\HH}),
\end{align*}
and the formula follows.
\end{proof}


\section{A scaling formula for cocentral extensions of discrete quantum groups}
\label{sec:scaling}

In this section we generalise the considerations of \cite[Section 2.1]{kyra}, which provides us with a scaling formula for $\ltwo$-Betti numbers of cocentral extensions of discrete quantum groups by abelian groups.  We start by collecting the analogues of \cite[Lemma 2.1 \& 2.2]{kyra}.
\begin{lemma}
  \label{lem:module-action}
  Let $A \stackrel{\veps}{\ra} \CC$ be an augmented algebra and $\Gamma \grpaction{\alpha} A$ an action of a finite group.  Consider the $A$-module $\cont(\Gamma)$ that is induced by the homomorphism $A \ni a \mapsto (g \mapsto \veps \circ \alpha_g(a))\in \cont(\Gamma)$.  Then
  \begin{equation*}
    \lmod{A}{\cont(\Gamma)}
    \cong
    \bigoplus_{g \in \Gamma} \lmod{\veps \circ \alpha_g}{\CC}
    \eqstop
  \end{equation*}
\end{lemma}
Here $ \lmod{\veps \circ \alpha_g}{\CC}$ denotes $\CC$ considered as an $A$-module via the homomorphism $\veps \circ \alpha_g\colon A \to \CC$. More generally, whenever $\beta\colon A\to B$ is a ring homomorphism and $X$ is a $B$-module, we denote by $ \lmod{\beta}{X}$ the $A$-module $X$ with module structure defined via $\beta$.

\begin{proof}
  The natural direct sum decomposition
  \begin{equation*}
    \cont(\Gamma)
    =
    \bigoplus_{g \in \Gamma} \CC \mathbb 1_{\{g\}}
    \end{equation*}
  is compatible with the $A$-module structure, since $\cont(\Gamma)$ is abelian.  Because $\lmod{A}{\CC \mathbb 1_{\{g\}}} \cong \lmod{\veps \circ \alpha_g}{\CC}$, we can conclude the lemma.
\end{proof}

\begin{remark}\label{rem:back-and-forth}
If $A$ is  Hopf $*$-algebra and $\Gamma$ is an abelian finite group, one can start out with a Hopf $*$-algebra homomorphism $\pi \colon A \ra \CC \Gamma$ and consider the induced action $\hat \Gamma \curvearrowright A$ (cf.~Section \ref{sec:preliminaries}).  Then the homomorphism $ A\ra \cont(\hat\Gamma)$ constructed in Lemma \ref{lem:module-action} coincides with $\pi$ after applying the Fourier transform $\CC \Gamma \cong \cont(\hat \Gamma)$.
\end{remark}

\begin{lemma}
  \label{lem:dim-tor-invariant}
  Let $A \stackrel{\veps}{\ra} \CC$ be an augmented $*$-algebra, $\alpha \in \Aut(A)$ and $\vphi \in A^*$ an $\alpha$-invariant tracial state with bounded GNS-representation.  Let $M = \pi_{\vphi}(A)''$.  Then for all $p \geq 0$
  \begin{equation*}
    \dim_M \Tor_p^A(M, \lmod{\alpha}{\CC})
    =
    \dim_M \Tor_p^A(M, \CC)
    \eqstop
  \end{equation*}
\end{lemma}
\begin{proof}
  A flat base change \cite[Proposition 3.2.9]{weibel} gives an isomorphism of $M$-modules
  \begin{equation*}
    \Tor_p^A(M, \lmod{\alpha}{\CC})
    \cong
    \Tor_p^A(M \ot_A (\lmod{\alpha}{A}), \CC)
    \eqstop
  \end{equation*}
  Note that $\alpha$ extends to an automorphism of $M$, which provides an isomorphism of left $M$-modules
  \begin{equation*}
    M  \ot_A (\lmod{\alpha}{A})
    \cong
    \lmod{\alpha}{M}
    \eqstop
  \end{equation*}
  Hence 
  \begin{equation*}
    \dim_{M} \Tor_p^A (M,  \lmod{\alpha}{\CC})
    =
    \dim_{M} \Tor_p^A (\lmod{\alpha}{M}, \CC)
    =
    \dim_{M} \lmod{\alpha}{ \bigl (\Tor_p^{A}(M, \CC) \bigr )}.
  \end{equation*}
  The endofunctor $X \mapsto  \lmod{\alpha}{X}$ on the category of left  $M$-modules preserves the class of finitely generated projective modules, is dimension preserving on this class and preserves inclusions, hence $\dim_{M}(X)= \dim_{M}( \lmod{\alpha}{X})$ for all $M$-modules $X$ (see \cite[Section 6.1]{Luck02}).  Therefore
  \begin{equation*}
    \dim_M \Tor_p^A (M,  \lmod{\alpha}{\CC})
    =
    \dim_M \Tor_p^A (M, \CC).
  \end{equation*}
  as claimed.
\end{proof}

\begin{proof}[Proof of Theorem \ref{thm:L2exact}]
 We have 
 \begin{align*}
   \beta_p^{(2)}(\hat{\HH})
   & =
     \dim_{\Linfty(\HH)}{\rm Tor}_p^{\Pol(\HH)}(\Linfty(\HH), {\CC}) 
    =
     \dim_{\Linfty(\GG)} \Linfty(\GG) \ot_{ \Linfty(\HH)}{\Tor}_p^{\Pol(\HH)}(\Linfty(\HH), \CC)
     \eqcomma
\end{align*}
since the functor $\Linfty(\GG) \ot_{\Linfty(\HH)} -$ is dimension preserving \cite[Theorem 3.18]{sauer-thesis}.  This functor is furthermore exact$^3$\footnotetext{$^3$For exactness and dimension-preservation in the case of group von Neumann algebras, see \cite[Theorem 6.29]{Luck02}} \cite[Theorem 1.48]{sauer-thesis}, and therefore commutes with $\Tor$, so that
\begin{equation*}
  \dim_{\Linfty(\GG)}\Linfty(\GG) \otimes_{\Linfty(\HH)}{\rm Tor}_p^{\Pol(\HH)}(\Linfty(\HH), \CC)
  =
  \dim_{\Linfty(\GG)}{\rm Tor}_p^{\Pol(\HH)}(\Linfty(\GG), \CC)
  \eqstop
\end{equation*}
Since the inclusion $\Pol(\HH) \subset \Pol(\GG)$ is flat by \cite{chi}, we can apply the flat base change formula \cite[\mbox{Proposition 3.2.9}]{weibel}, which gives
\begin{equation*}
  \dim_{\Linfty(\GG)} \Tor_p^{\Pol(\HH)}(\Linfty(\GG), \CC)
  =
  \dim_{\Linfty(\GG)} \Tor_p^{\Pol(\GG)}(\Linfty(\GG), \Pol(\GG) \ot_{\Pol(\HH)} \CC)
  \eqstop
\end{equation*}
The exactness assumption on our sequence gives in particular 
\begin{equation*}
\Pol(\GG) \ot_{\Pol(\HH)} \CC \cong \Pol(\GG) / \Pol(\GG) \Pol(\HH)^+ \cong \CC \Gamma \cong \cont(\hat \Gamma),
\end{equation*}
as left $\Pol(\GG)$-modules where the $\Pol(\GG)$-structure on $ \cont(\hat \Gamma)$ is defined via the identification with $\CC\Gamma$. 
The theorem now follows  from a combination of Lemmas \ref{lem:module-action}, \ref{lem:dim-tor-invariant} and Remark \ref{rem:back-and-forth} together with the fact that the dimension function is additive.
\end{proof}

\section{Calculations of $\ltwo$-Betti numbers}
\label{sec:calculations}

In this section we apply the formulas obtained in Sections \ref{sec:free-product} and \ref{sec:scaling} to the specific examples of cocentral extensions presented in \cite{bny, bny2}. This will provide a complete calculation of the $\ltwo$-Betti numbers of the universal discrete quantum groups $\hat \rU_n^+$  (all but the second and third were already found in \cite{kyra}) and, furthermore, a complete calculation of the $\ltwo$-Betti numbers of the duals of the half-liberated unitary quantum groups $\rU_n^*$.

\begin{proof}[Proof of Theorem \ref{thm:intro::ltwo-betti-numbers-un}]
  As explained in Section \ref{sec:preliminaries}, we have exact sequences of Hopf $*$-algebras
  \begin{gather*}
    \CC \lra \Pol(\HH) \lra  \Pol(\rU_n^+) \lra \CC \ZZ_2 \lra \CC \\
    \CC \lra \Pol(\HH) \lra \Pol(\rO_n^+) * \Pol(\rO_n^+) \lra \CC \ZZ_2 \lra \CC,
  \end{gather*}
for the same compact quantum group $\HH$.  Applying Theorem \ref{thm:intro:free-product-formula} and \ref{thm:L2exact} this gives
\begin{equation*}
  \beta_p^{(2)}(\hat \rU_n^+)
  =
  \frac{1}{2} \beta_p^{(2)}(\hat \HH)
  =
 \beta_p^{(2)}(\hat \rO_n^+ * \hat \rO_n^+)
  =
  \begin{cases}
    0 & p = 0 \\
    2 \cdot \left(\beta_1^{(2)}(\hat \rO_n^+) - \beta_0^{(2)}(\hat \rO_n^+)\right) + 1 & p = 1 \\
    2 \cdot \beta_p^{(2)}(\hat \rO_n^+) & p \geq 2
    \eqstop
  \end{cases}
\end{equation*}
Since $\beta_p^{(2)}(\hat \rO_n^+) = 0$ for all $p \geq 0$ by \cite{vergnioux-paths-in-cayley} and \cite{thom-collins}, Theorem \ref{thm:intro::ltwo-betti-numbers-un} follows.
\end{proof}

\begin{proof}[Proof of Theorem \ref{thm:intro:other-examples}]
By \cite[Example 3.7]{bny}, there exist short exact sequences of Hopf $*$- algebras
\begin{gather*}
    \CC \lra \Pol(\HH^*) \lra  \Pol(\rU_n^*) \lra \CC \ZZ_2 \lra \CC \\
    \CC \lra \Pol(\HH^*) \lra \Pol(\rO_n^*) * \Pol(\rO_n^*) \lra \CC \ZZ_2 \lra \CC,
  \end{gather*}
for the same compact quantum group $\HH^*$; the details of this argument are analogous to those sketched in Section \ref{sec:preliminaries} for the free unitary quantum groups. By Theorem \ref{thm:intro:free-product-formula} and \ref{thm:L2exact} we therefore obtain

\begin{equation*}
  \beta_p^{(2)}(\hat \rU_n^*)
  =
  \frac{1}{2} \beta_p^{(2)}(\hat \HH^*)
  =
 \beta_p^{(2)}(\hat \rO_n^* * \hat \rO_n^*)
  =
  \begin{cases}
    0 & p = 0 \\
    2 \cdot \left(\beta_1^{(2)}(\hat \rO_n^*) - \beta_0^{(2)}(\hat \rO_n^*)\right) + 1 & p = 1 \\
    2 \cdot \beta_p^{(2)}(\hat \rO_n^*) & p \geq 2
    \eqstop
  \end{cases}
\end{equation*}
However, since $O_n^*$ is infinite and coamenable \cite[Corollary 9.3]{banica-vergnioux}, $\beta_p^{(2)}(\hat{\rO}_n^*)=0$ for all $p\geq 0$ by \cite{coamenable-betti} and the result follows. 
\end{proof}


\def\cprime{$'$} \def\cprime{$'$}


\vspace{0.5em}
\begin{minipage}{1.0\linewidth}
  \uppercase{David Kyed},
  \newblock Department of Mathematics and Computer Science, University of Southern Denmark,
  \newblock Campusvej 55,
  \newblock DK-5230 Odense M, Denmark. \\
  \newblock E-mail address: \textit{dkyed@imada.sdu.dk}
\end{minipage}

\vspace{1em}
\begin{minipage}{1.0\linewidth}
  \mbox{\uppercase{Julien Bichon}},
  \mbox{Laboratoire de Math{\'e}matiques Blaise Pascal},
  \mbox{Universit{\'e} Clermont Auvergne},
  \mbox{Campus universitaire des C{\'e}zeaux},
  \mbox{3 place Vasarely},
  \mbox{ 63178 Aubi{\`e}re cedex, France}. \\
  \mbox{E-mail address: \textit{julien.bichon@uca.fr}}
\end{minipage}

\vspace{1em}
\begin{minipage}{1.0\linewidth}
  \uppercase{Sven Raum},
  \newblock EPFL SB SMA,
  \newblock Station 8,
  \newblock CH-1015 Lausanne, Switzerland \\
  \newblock E-mail address: \textit{sven.raum@epfl.ch}
\end{minipage}

  
 

\end{document}